\documentclass[10pt]{amsart}
\usepackage{amssymb}
\usepackage{epsfig}
\usepackage{url}
\usepackage{setspace}
\usepackage{color}
\theoremstyle{plain}

\newtheorem{thm}{Theorem}[section]
\newtheorem{cor}[thm]{Corollary}

\newtheorem{rem}[thm]{Remark}
\newtheorem{ques}[thm]{Question}

\newtheorem{prob}[thm]{Problem}

\theoremstyle{remark}
\newtheorem*{acknowledgement}{Acknowledgement}

\def\bbb{\mathbb}
\def\op{\operatorname}
\renewcommand{\phi}{\varphi}

\newcommand{\N}{\bbb{N}}
\newcommand{\Z}{\bbb{Z}}
\newcommand{\Q}{\bbb{Q}}

\usepackage{longtable}

\makeatletter
\let\@@pmod\pmod
\DeclareRobustCommand{\pmod}{\@ifstar\@pmods\@@pmod}
\def\@pmods#1{\mkern4mu({\operator@font mod}\mkern 6mu#1)}
\makeatother

\begin{document}
\bibliographystyle{plain}
\title{Diophantine problems related to cyclic cubic and quartic fields}
\author{Szabolcs Tangely and Maciej Ulas}

\keywords{} \subjclass[2010]{}
\thanks{The research of the first author was supported in part by the NKFIH grants 115479, 128088 and 130909 and by the project EFOP-3.6.1-16-2016-00022, co-financed by the European Union and the European Social Fund. The research of the second author is supported by the grant of the National Science Centre (NCN), Poland, no. UMO-2019/34/E/ST1/00094}

\begin{abstract}
We are interested in solving the congruences $f^3+g^3+1\equiv 0\pmod{fg}$
and
$f^4-4g^2+4\equiv 0\pmod{fg}$ in polynomials $f, g$ with rational coefficients. Moreover, we present results of computations of all integer points
on certain one parametric curves of genus 1 and 3, related to cubic and quartic fields, respectively.
\end{abstract}

\maketitle

\section{Introduction}\label{sec1}
A family of cubic fields studied by Cohn \cite{CohnCubic}, Ennola \cite{EnnolaCubic} and Shanks \cite{ShanksSC} is related to the parametric polynomial
$$
X^3 + (n + 3)X^2 + nX - 1=h_{n}(X).
$$
The family $(h_{n}(X)){n\in\N}$ has been further investigated by many authors e.g. by  Lecacheux \cite{LecaCubic}, Lettl \cite{LettlCubic} and Washington \cite{WashC1}.
Kishi \cite{Kishi} and Washington \cite{WashC2} provided other one-parameter families with similar properties. Later, Balady \cite{Bal3,BalPhD} generalized Washington's procedure \cite{WashC2} and described a method for
generating other one-parameter families. In this construction polynomials
$f$ and $g$ having integral coefficients such that the fraction
$$
\frac{f(t)^3+g(t)^3+1}{f(t)g(t)}=:\lambda(t)
$$
is also a polynomial with integral coefficients play an important role. More precisely, if $\lambda\in \Z[t]$, then the pair of polynomials $(f, g)$ determines one-parameter family of polynomials
$$
P_{f,g}(X)=X^{3}+a(n)X^2+\lambda(n)X-1,
$$
where $a(n)=3(f(n)^2+g(n)^2-f(n)g(n))-\lambda(n)(f(n)+g(n))$. For all but finitely many values of $n$, the polynomial $P_{f,g}$ has three real roots which are units and more importantly the field generates by these roots is cyclic cubic field.

These families turned out to be interesting from Diophantine number theory point of view as well. Indeed, for any given $n$, the equation $x^3+y^3+\lambda(n)xy+1=0$, where $n$ is an integer parameter, defines a curve in cubic Hessian form and defines genus 1 curve (and thus can be seen as a one possible model of an appropriate elliptic curve). Moreover, it satisfies Runge type condition and thus, for any given $n$, one can offer an efficient algorithm to finding all integral points without necessity of computing generators of the Mordell-Weil group of an associate elliptic curve.

On the other side, the equation $Y^3h_{n}(X/Y)=\pm 1$ associated with the simplest cubic fields of Shanks, defines a (parametric) Thue type equation and it is natural to ask about characterization of its integral solutions. In this direction one may consult the results \cite{LPV, LevWal, MPL, Thomas}. For any given $n$, one can also consider the equation $h_{n}(X)=Y^2$ as the equation defining an elliptic curve and mention \cite{DuqSC}, where questions concerning the existence of integer/rational points, characterization of torsion groups and ranks are investigated. It is clear that analogous question can be asked to the more general family $(P_{f(n),g(n)})_{n\in\N}$.

Similarly, families of cyclic quartic fields with explicit units have been studied by many authors (see e.g. \cite{Gras,LecaQ,LecaCubic}). Based on ideas given in the paper \cite{Bal3} Balady and Washington \cite{BalWas4} studied cyclic quartic fields. They considered the polynomial
$$
X^4+(2s^3+Ls^2-4ps+2Lp)X^3-(3s^2+3Ls-6p)X^2+2LX+1,
$$
where
$$
L=-\frac{s^4-4p^2+4}{4 s p}.
$$
If $L$ or $2L\in\Z$ then the roots of the above polynomial are units. Under suitable specialization we get a quartic number field with cyclic Galois group of order 4 and the roots generate the group of units or a subgroup of index 5. We thus see that the necessary condition to get something interesting is the investigation of the divisibility condition $4fg|(f^4-4g^2+4)$ which can be seen as a quartic analogue of the condition $fg|(f^3+g^3+1)$ considered in a cubic case.

The findings by Shanks, Balady and Balady and Washington motivated us to look for solutions of the congruences
\begin{equation}\label{maincong}
f^3+g^3+1\equiv 0\pmod*{fg}
\end{equation}
and
\begin{equation}\label{quartic}
f^4-4g^2+4\equiv 0\pmod*{fg}
\end{equation}
where $f, g$ are polynomials with rational coefficients. Note that if $f,
g$ are treated as integers, then from the recent results of Schinzel \cite{Sch1} (which extend earlier results of Mordell \cite{Mor1}) we know that each congruence above has infinitely many solutions. Thus, the question
concerning the existence of polynomial solutions is quite natural in this
direction too. It is clear that finding polynomial solutions is more difficult question. Moreover, before we start let us observe that if $f, g\in\Z[t]$ is a solution, then for each $h\in\Z[t]$ the polynomials $f\circ h, g\circ h$ is also a solution. Thus, our first attempt to remedy this situations is to looking solutions of the following form
\begin{equation}\label{fg}
f(t)=\sum_{i=0}^{m-2}a_{i}t^{i}+a_{m}t^{m},\quad g(t)=\sum_{i=0}^{n}b_{i}t^{i}.
\end{equation}
In other words, instead of asking for $f, g$ with integer coefficients we
are asking for $f, g$ with rational coefficients. Here, as usual, $m=\op{deg}f, n=\op{deg}g$, i.e., $a_{m}b_{n}\neq 0$. In order to deal with the case of substitutions $h$ of the form $h(t)=pt, p\in\Q\setminus\{0\}$ or $h\in\Q[t]$ is of degree $\geq 2$, we will be interested only in the {\it non-trivial} solutions. More precisely, we say that a pair of polynomials $f, g\in\Q[t]$ of the form given above, is a non-trivial solution
of the congruence (\ref{maincong}) (or the congruence (\ref{quartic})) if
and only if:
\begin{enumerate}
\item[(a)] there are no polynomials $f_{1}, g_{1}\in\Q[t]$ and the polynomial $h\in\Q[t]$ of degree $\geq 2$ such that $f=f_{1}\circ h, g=g_{1}\circ h$;
\item[(b)] the leading coefficient $a_{m}$ of the polynomial $f$ is free of $m$-th power of a rational integer. In other words, for each prime $p$, we have $|\nu_{p}(a_{m})|<m$, where $\nu_{p}(q_{1}/q_{2})$ is the $p$-adic valuation of the rational number $q_{1}/q_{2}$.
\end{enumerate}

Let us describe the content of the paper in some details. In Section \ref{sec2} and Section \ref{sec3} we investigate congruences (\ref{maincong}), (\ref{quartic}), respectively. In particular, we characterize all polynomial solutions $(f, g)$ satisfying $\deg f\leq 2$. In case of the congruence (\ref{quartic}) we also characterize solutions with $\deg f=3, \deg g=2$.

In Section \ref{sec4} we investigate the existence and characterizations of all integer points on the curves $H_{d}:\;x^3+y^3+dxy+1=0$ form computational point on view. In fact, for any $d\in [-2\cdot 10^6, 2\cdot 10^6]$ we found all integer points on $H_{d}$. A similar investigations are presented in the case of genus three curves $X_{T}:\; (x+y)^4-4x^2y^2+Txy(x+y)+4=0$, where $T$ is even and $2\leq T\leq 15\cdot 10^6$ and $Q_{T}:\;x^4-4y^2+4Txy+4=0$ for $T\leq 5\cdot 10^6$.
 
\section{Cubic case}\label{sec2}

It is clear that the study of the existence of solutions of the congruence (\ref{maincong}) is equivalent with studying the existence of solutions
of the system of congruences
\begin{equation}\label{main}
f^{3}+1\equiv 0\pmod{g},\quad g^{3}+1\equiv \pmod{f}.
\end{equation}
Moreover, due to symmetric nature of (\ref{maincong}) without loss of generality we can assume that $m\leq n$. As an immediate consequence of the equivalent formulation we get that $n\leq 3m$. In other words, if we fix the degree of a polynomial $f$ then the degree $n$ of $g$ is bounded by $3m$.

Let
$$
\bar{A}=(a_{0}, a_{1}, \ldots, a_{m-2}, a_{m}), \quad \bar{B}=(b_{0},
b_{1}, \ldots, b_{n-1}, b_{n})
$$
be the vectors of variables. We define the $F_{i}=F_{i}(\bar{A},\bar{B}), i=0,\ldots n-1$ and $G_{i}(\bar{A},\bar{B}), i=0,\ldots, m-1$ as the numerators of the coefficients in the remainders of divisibility of $f^3+1\pmod{g}$ and $g^3+1\pmod{f}$ respectively. More precisely,
\begin{align*}
(f(t)^3+1)\pmod{g(t)}&=\sum_{i=0}^{n-1}F_{i}(\bar{A},\bar{B})t^{i},\\
(g(t)^3+1)\pmod{f(t)}&=\sum_{i=0}^{m-1}G_{i}(\bar{A},\bar{B})t^{i},
\end{align*}
i.e., the above congruences are unique up to multiplication by appropriate power of $b_{n}$ and $a_{n}$ respectively.

Our first aim is to attack the congruence (\ref{maincong}) in a systematic way. More precisely, we are interested in finding all non-trivial solutions $f, g$ of (\ref{maincong}) in the case when $\op{deg}f\leq 2$. It is
clear that in order to find the pairs of polynomials satisfying the system (\ref{main}) we need to study rational solutions of the system
\begin{equation*}
S(m,n):\;\begin{cases}\begin{array}{cc}
                        F_{i}(\bar{A},\bar{B})=0, & \quad i=0,\ldots,
n-1, \\
                        G_{j}(\bar{A},\bar{B})=0, & \quad j=0,\ldots,
m-1.
                      \end{array}
\end{cases}
\end{equation*}

Let us note that for given $m, n$, the system $S(m,n)$ is defined by $m+n$ equations in $m+n+1$ variables. We could expect, that in general $S(m,n)$ defines a curve. This is not the case due to the problem with trivial solutions. More precisely, let $d$ be a divisor of $\gcd(m,n)$ and observe that with any given solution $f, g$ of the system $S(m/d,n/d)$ we get the solution $f\circ h, g\circ h$ of system $S(m,n)$, where $h$ is of degree $d$. Thus in this case, in some sense we will get $d$ dimensional set of trivial solutions. In particular, because of $d=1$ we will always have a one dimensional set of trivial solutions. The problem is that the system ``doesn't seem" the difference between trivial and non-trivial solutions. Thus, we expect that essentially we have $m+n$ true variables (which is the case of $\gcd(m,n)=1$). In other words, the set of non-trivial
solutions in case of fixed value of $(m,n)$ is finite. Thus, the strategy, at least in theory, is clear. Compute the Gr\"{o}bner basis of the ideal generated by the polynomials defined the system $S(m,n)$ and extract for it all non-trivial solutions. It is clear that the necessary and sufficient condition for solvability of $S(m,n)$ is the vanishing of each polynomial in the Gr\"{o}bner basis. Unfortunately, the corresponding ideals are quite big and the computations start to be difficult very fast. This is the reason, why we will deal with small values of $m=\op{deg}f$, i.e.,  with $m\leq 2$. One can think that our aim is modest, but as we will see even in such relatively small cases of $\deg f$, the corresponding systems of equations are quite big and difficult to solve. Note that the same discussion applies also to the congruence (\ref{quartic}).

All Gr\"{o}bner basis computations were performed with the Mathematica \cite{math}. The program was running on a laptop with 32 GB of RAM memory and i7 type processor. We used standard lexicographic order in our computations, i.e.,
$$
b_{n}<b_{n-1}<\ldots b_{1}<b_{0}<a_{m}<a_{m-2}<\ldots <a_{1}<a_{0}.
$$

All the systems considered in this section are presented in the full form in Appendix A.  

\subsection{The case $\deg f=1$}\label{subcubic1}

Without loss of generality we can write $f(t)=t$. Thus $g(t)$ is a divisor of
$$
f(t)^3+1=(t+1)(t^2-t+1).
$$
Thus, the polynomial $g$ takes one of the following form: $g(t)=b_{1}(t+1)$ or $g(t)=b_{2}(t^2-t+1)$ or $g(t)=b_{3}(t^3+1)$. A quick inspection reveals that $b_{1}=-1$ or $b_{2}=-1$ or $b_{3}=-1$. Thus, the only non-trivial solutions of the congruence (\ref{maincong}) in case of $\deg f=1$ are given by
\begin{equation}\label{m1}
\begin{array}{lll}
  f(t)=t, & g(t)=-t-1,     & \lambda(t)=3,\\
  f(t)=t, & g(t)=-t^2+t-1, & \lambda(t)=t^3-2 t^2+3 t-3\\
  f(t)=t, & g(t)=-t^3-1,   & \lambda(t)=t^2 \left(t^3+2\right).
\end{array}
\end{equation}

In particular, in case of the system $S(m,im)$ with $i=1,2,3$ we will always have trivial solutions of the form $f\circ h, g\circ h$, where $h$ is a general polynomial of degree $m$ and the pair $(f, g)$ is the one from (\ref{m1}).

In the following subsection we offer systematic analysis of the rational solutions of the system $S(2,n)$ for $n\in\{2,\ldots,6\}$.

\subsection{The case $\deg f=2, \deg g=2$}\label{subcubic22}

Here we investigate rational solutions of the polynomial system $S(2,2)$.

First we consider the case $b_{1}=0$. If $b_{1}=0$ then the polynomials $F_{1}, G_{1}$ vanish and we deal with the system
$$
S'(2,2):\;P_{1}P_{2}=0,\quad Q_{1}Q_{2}=0,
$$
where

$$
\begin{array}{ll}
P_{1}=a_2 b_0-a_0 b_2-b_2, & P_{2}=(b_0^2+b_{2}^2)a_2^2-2 a_0 a_2 b_2
b_0+a_2 b_2 b_0-a_0b_2^2+b_2^2,\\
Q_{1}=a_2 b_0-a_0 b_2+a_2, &       Q_{2}=(b_0^2+b_{2}^2)a_2^2-a_2^2 b_0+a_0 a_2 b_2-2 a_0 a_2 b_0 b_2+a_2^2.
\end{array}
$$

In order to solve the system $S'(2,2)$ it is enough to solve the four systems $P_{i}=Q_{j}=0$ for $i, j\in\{1,2\}$.

If $P_{1}=Q_{1}=0$,  then $a_{0}=-b_{0}-1, b_{2}=-a_{2}$. This leads to the trivial solution which is re-parametrization of the first solution in (\ref{m1}) with $t$ replaced by $a_{2}t^2-b_{0}-1$.

If $P_{1}=0, Q_{2}=0$, then from the first equation we get $a_{0}=(a_2 b_0-b_2)/b_2$. Putting this into the second one we get the equation $a_2^2-a_{2}b_{2}+b_2^2=0$ without non-zero solutions in rationals. In exactly the same way we deal with the system $P_{2}=0, Q_{1}=0$.

Finally, if $P_{2}=0, Q_{2}=0$, then we observe that $P_{2}-Q_{2}=(a_2+b_2)(a_2 b_0-a_0 b_2-a_2+b_2)$. If $a_{2}=-b_{2}$ then we get the equation $(a_0+b_0)^2-(a_0+b_0)+1=0$ without solutions in rationals. If $a_{2}+b_{2}\neq 0$, then we need to have $b_{0}=(a_0 b_2+a_2-b_2)/a_2$
and our system reduces to the equation $a_2^2-a_{2}b_{2}+b_2^2=0$ without non-zero solutions in rationals.

Now let $b_{1}\neq 0$. In order to get the solutions of the system $S(2,2)$ we compute the Gr\"{o}bner basis, say $G(2,2)$, of the ideal generated
by the set of polynomials
$\{F_{0}, F_{1}/b_{1}, G_{0}, G_{1}/b_{1}\}$.  The set $G(2,2)$ contains polynomials $H_{i}, i=1,\ldots, 25$. We have that
$$
\begin{array}{ll}
H_{1}=(8a_0-3)(64 a_0^2+24 a_0+9)a_2^5, &   H_{2}=a_2^5(8 b_0-7)(64b_0^2+56b_0+49),\\
H_{3}=-a_2^4(8192a_0^2a_2b_0^2-441b_1^2), &       H_{8}=-a_2^2(2a_0b_1^4-3a_2b_2).
\end{array}
$$

Let us recall that $a_{2}\neq 0$. We thus solve the triangular (with respect to the variables $a_{0}, b_{0}, a_{2}, b_{2}$) system defined by the vanishing of the polynomials from the set $H$ and get
$$
a_0=\frac{3}{8},\quad b_0=\frac{7}{8},\quad a_2=\frac{b_1^2}{2},\quad b_{2}=\frac{b_1^2}{2}.
$$
In fact, the expressions for $a_{0}, b_{0}, a_{2}, b_{2}$ given above solve also the whole system $S(2,2)$. We thus get the solution
$f(t)=(4 b_1^2 t^2+3)/8, g(t)=(4 b_1^2 t^2+8 b_1 t+7)/8$,
which after the re-parametrization $t\rightarrow (2t-1)/2b_{1}$, takes the form
\begin{equation}\label{sol22}
f(t)=\frac{1}{2}(t^2-t+1),\quad g(t)=\frac{1}{2}(t^2+t+1)=f(-t).
\end{equation}
We also have $\lambda(t)=t^2+5$.

We thus proved

\begin{thm}\label{22case}
The only non-trivial solution of the system {\rm (\ref{main})} with $\op{deg}f=\op{deg}g=2$ is the one given by {\rm (\ref{sol22})}.
\end{thm}

Let us also note

\begin{cor}
The only (non-trivial) solution of the congruence $f^3+g^3+8\equiv 0\pmod{fg}$ with $\op{deg}f=\op{deg}g=2$ is of the form $f(t)=t^2-t+1, g(t)=f(-t)$.
\end{cor}
In the above the notion of a non-trivial solution of the congruence $f^3+g^3+8\equiv 0\pmod{fg}$ is understood in exactly the same way as in the case of the congruence $f^3+g^3+1\equiv 0\pmod{fg}$.

\subsection{The case $\deg f=2, \deg g=3$}\label{sub23}

We are interested in solving the system (\ref{main}) in polynomials $f, g\in\Q[t]$, with $\op{deg}f=2, \op{deg}g=3$.  Equivalently, we are interested in the rational solutions of the system $S(2,3)$.

We computed the Gr\"{o}bner basis, say $G(2,3)$, of the set of polynomials defining the system $S(2,3)$. The set $G(2,3)$ contains the polynomials
$H_{i}, i=1, \ldots, 451$. We have
\begin{align*}
H_{1}=&b_2 b_3^4(8 b_2^3+b_3^2)(64 b_2^6-280 b_3^2 b_2^3-243 b_3^4)(64 b_2^6-8 b_3^2 b_2^3+b_3^4)\times\\
      &(4096 b_2^{12}+17920 b_3^2 b_2^9+93952 b_3^4 b_2^6-68040 b_3^6 b_2^3+59049 b_3^8),\\
H_{2}=&b_3^4(8 b_2^3+b_3^2)(64 b_2^6-8 b_3^2 b_2^3+b_3^4)(512 b_2^{11}-96218 b_3^6 b_2^2+42939 b_1 b_3^7),\\
H_{5}=&b_3^4(-23396352 b_2^{13}+35946496 b_1 b_3 b_2^{11}+43985895 b_3^6 b_2^4+\\
      &\quad -11830222 b_1 b_3^7 b_2^2+1190043 b_1^2 b_3^8).
\end{align*}
We consider two cases: $b_{2}\neq 0$ or $b_{2}=0$.

Let $b_{2}\neq 0$. We are interested in rational solutions of $H_{1}=H_{2}=H_{5}=0$ (together with the condition $b_{3}\neq 0$), from the first two equations we get that $8 b_2^3+b_3^2=0$. This is equivalent with $b_{2}=-2 u^2,\quad b_{3}=8u^3$ for some $u\in\Q$. However, substituting the computed values in the equation $H_{5}=0$, we left with the equation $C_{0} u^{36} \left(-20 b_1 u+4 b_1^2+37 u^2\right)=0$, where $C_{0}$ is a non-zero rationals. This equation has a rational solution only for $u=0$. We get a contradiction because $b_{3}\neq 0$.

If $b_{2}=0$, from the equation $H_{2}=0$ we get that $b_{1}=0$. Next, using the computed values we get that
$$
H_{53}=C_{1}(b_0+1)(b_0^2-b_0+1)b_3^8,\quad H_{120}=C_{2}b_3^5(b_0 b_3^2-a_2^3),\quad H_{261}=C_{3} a_0 a_2 b_3^6,
$$
where $C_{1}, C_{2}, C_{3}$ are non-zero rationals. We get that
$$
b_{0}=-1,\quad b_{3}=u^3,\quad a_{2}=-u^2, \quad a_{0}=0,
$$
for some $u\in\Q$. The solution obtained in this way is $f(t)=-t^2 u^2, g(t)=-1 + t^3 u^3$. After the reparametrization $t\rightarrow t/u$, it takes the form
\begin{equation}\label{sol23first}
f(t)=-t^2,\quad g(t)=t^3-1.
\end{equation}
We also have $\lambda(t)=-t(t^3-3)$.

We thus proved

\begin{thm}\label{23case}
The only non-trivial solution of the system {\rm(\ref{main})} with $\op{deg}f=2, \op{deg}g=3$ is the one given by {\rm (\ref{sol23first})}.
\end{thm}

\subsection{The case $\deg f=2, \deg g=4$}\label{sub24}

In this section we look for polynomial solutions of the system $S(2,4)$. Let us note that the set of rational solutions of the system $S(2,4)$ is the set theoretic sum of the set of rational solutions of the systems $S_{1}(2,4)$ and $S_{2}(2,4)$, where
\begin{align*}
S_{1}(2,4):\;F_{0}=F_{1}=F_{2}=G_{0}=G_{1,1}=0,\\
S_{2}(2,4):\;F_{0}=F_{1}=F_{2}=G_{0}=G_{1,2}=0,
\end{align*}
where $G_{1,1}=a_2b_1-a_0b_3$ and  $G_{1,2}=G_{1}/G_{1,1}$.

Let $G_{i}(2,4)$ be the Gr\"{o}bner basis of the set of polynomials defining the systems $S_{i}(2,4)$ for $i=1, 2$.

We consider the case $i=1$. Then $G_{1}(2,4)$ contains polynomials $H_{i}, i=1, \ldots, 110$. We have $H_{1}=b_3 b_4^3, H_{2}=b_1 b_4^3$ and
\begin{align*}
H_{4}=&b_4^4(b_2^2-4 b_0 b_4-3 b_4)(b_2^4-8 b_0 b_4 b_2^2+3 b_4 b_2^2+16 b_0^2 b_4^2-12 b_0 b_4^2+9 b_4^2),\\
H_{63}=&b_4^2(a_2^2+b_4)(-a_2^2 b_4+a_2^4+b_4^2),\\
H_{72}=&9(2 a_0 b_4-a_2b_2)b_4^4-(b_4 b_2^4-b_3^2 b_2^3-8 b_0 b_4^2 b_2^2+16 b_0^2 b_4^3)a_2^4.
\end{align*}
The unique rational solution of the system $H_{1}=H_{2}=H_{4}=H_{63}=H_{72}$ with the respect to the variables $b_{0}, b_{1}, b_{2}, b_{3}, b_{4}$ is
$$
b_{0}=-a_0^2+a_0-1,\quad b_{1}=0,\quad b_{2}=(1-2 a_0) a_2,\quad  b_{3}=0, \quad b_{4}=-a_{2}^2.
$$
However, with $b_{i}, i=0, \ldots, 4$ chosen in this way we get that $g(t)=-(a_2t^2+a_0)^2+(a_2t^2+a_0)-1$, i.e., trivial solution coming from
the second solution in (\ref{m1}) with $t$ replaced by $a_2t^2+a_0$.

We consider the case $i=2$ now. Then $G_{2}(2,4)$ contains polynomials $H_{i}, i=1, \ldots, 479$. We have
\begin{align*}
H_{1}=&b_4^3 \left(b_3^4-4 b_4^3\right) \left(b_3^8+4 b_4^3 b_3^4+16 b_4^6\right)&,\\
H_{2}=&-b_4^4 \left(9 b_3^2-8 b_2 b_4\right),\\
H_{11}=&-b_4^3 \left(11 b_3^3-16 b_1 b_4^2\right),\\
H_{28}=&-b_4^3 \left(77 b_3^4-256 b_0 b_4^3\right),\\
H_{163}=&-b_4^3 \left(8 b_2^2 b_3^2-81 a_2^3 b_4\right).
\end{align*}
From the equation $H_{1}=0$ we get that $b_{3}=4u^3, b_{4}=4u^4$ for some $u\in\Q\setminus\{0\}$. Then, one can easily solve the system $H_{2}=H_{11}=H_{28}=H_{163}=0$ and we get
$$
b_{0}=\frac{77}{64},\quad b_{1}=\frac{11 u}{4},\quad b_{2}=\frac{9 u^2}{2},\quad a_{2}=2u^2.
$$
With $b_{i}$ and $a_{2}$ chosen in this way the equation $H_{254}=0$ takes the form $Cu^{20} (-3 + 8 a_{0})=0$, where $C\in\Q\setminus\{0\}$, and thus $a_{0}=3/8$.
Finally, if we replace $t$ by $(2t-1)/4u$ we get a non-trivial solution
\begin{equation}\label{sol24}
f(t)=\frac{1}{2}(t^2-t+1),\quad g(t)=\frac{1}{4}(t^2+t+1)(t^2-t+3).
\end{equation}
We also have $\lambda(t)=\frac{1}{8} \left(t^6+t^5+6 t^4+9 t^3+18 t^2+21 t+33\right)$.

We thus proved the following

\begin{thm}\label{24case}
The only non-trivial solution of the system {\rm(\ref{main})} with $\op{deg}f=2, \op{deg}g=4$ is the one given by {\rm (\ref{sol24})}.
\end{thm}

\subsection{The case $\deg f=2, \deg g=5$}\label{sub25}

We deal with the system $S(2,5)$. First we consider the case $b_{4}=0$. Then the equations $F_{0}=F_{1}=F_{3}=0$ implies that $a_{0}=-1, b_{0}=b_{2}=0$. However, then the equation $G_{0}=0$ reduces to $a_{2}^{7}=0$ - a contradiction.

Let us suppose that $b_{4}\neq 0$. It is easy to see that the (sub)system consisting the equations $F_{i}=0$ for $i=1, 2, 3, 4$ is triangular with respect to the variables $b_{0}, b_{1}, b_{2}, b_{3}$. The required solution is given by
\begin{align*}
b_0=&-\frac{(a_0^3+1)b_5^2}{a_2^3 b_4},\\
b_1=&-\frac{(a_0^3+1)b_5^3}{a_2^3 b_4^2},\\
b_2=&-\frac{b_5^2 \left(a_0^3 b_5^2+3 a_2 a_0^2 b_4^2+b_5^2\right)}{a_2^3 b_4^3},\\
b_3=&-\frac{b_5^3 \left(a_0^3 b_5^2+3 a_2 a_0^2 b_4^2+b_5^2\right)}{a_2^3 b_4^4}
\end{align*}
With $b_{i}, i=0, 1, 2, 3$, we consider the remaining part of the system $S(2,5)$, i.e., $F_{4}=G_{0}=G_{1}=G_{2}=0$. It is easy to compute the Gr\"{o}bner basis, say $G(2,5)$ of the corresponding ideal. The basis contains polynomials $H_{i}, i=1, \ldots, 41$ and we have
\begin{align*}
H_{1}=&(a_0^3+1)^6(4 a_0+3)(16 a_0^2-12 a_0+9) a_2^{13} b_5^{16},\\
H_{5}=&(a_0^3+1)^6b_5^{16}(a_2^5+b_5^2)(-a_2^5 b_5^2+a_2^{10}+b_5^4),\\
H_{9}=&(a_0^3+1)^6a_2^7 b_5^{16}(8 a_0^2 a_2^7+9 b_4 b_5^2).
\end{align*}
If $a_{0}=-1$ then the equation $H_{13}=0$ reduces to $Ca_2^{16} b_4^2 b_5^{12}=0$, where $C\in\Q\setminus\{0\}$. Thus, $b_{4}=0$ and we back to the case we already considered on the beginning.

If $a_{0}\neq -1$, then the unique rational solution of the system $H_{1}=H_{5}=H_{9}=0$ is
$$
a_{0}=-\frac{3}{4},\quad a_{2}=u^{5},\quad b_{4}=-\frac{u^{31}}{2},\quad b_{5}=u^2.
$$
Finally, to get the solution of our initial system $S(2,5)$ we need to take $u=-1$. Tracing back our reasoning we get the obtained solution after the change of variables $t\rightarrow (2t-1)/2$, leads to a non-trivial
solution of the system (\ref{main})
\begin{equation}\label{sol25}
f(t)=-t^2+t-1,\quad g(t)=t(t^4-2 t^3+4 t^2-3 t+3).
\end{equation}
We also have $\lambda(t)=-t^8+3 t^7-8 t^6+11 t^5-15 t^4+10 t^3-8 t^2-1$.

We thus proved

\begin{thm}\label{25case}
The only non-trivial solution of the system {\rm(\ref{main})} with $\op{deg}f=2, \op{deg}g=5$ is the one given by {\rm (\ref{sol25})}.
\end{thm}

\subsection{The case $\deg f=2, \deg g=6$}\label{sub26}

This case is easy. Indeed, if $g\in\Q[t]$ is of degree 6, then the condition $f(t)^3+1\equiv 0\pmod*{g(t)}$ implies the equality $f(t)^3+1=cg(t)$ for some $c\in\Q\setminus\{0\}$. Because $f(t)=a_{2}t^2+a_{0}$, thus $b_{1}=b_{3}=b_{5}=0$ and comparing coefficients on both sides of the equality $f(t)^3+1=cg(t)$ we get that
$$
b_{0}=\frac{a_0^3+1}{c},\quad b_{2}=\frac{3 a_0^2 a_2}{c},\quad b_{4}=\frac{3 a_0a_2^2}{c},\quad b_{6}=\frac{a_2^3}{c}.
$$

The condition $g(t)^3+1\equiv 0\pmod*{f(t)}$ implies that $c=-1$. A quick computation reveals that $g(t)=-f(t)^3-1$. However, this is trivial solution coming form the third solution in (\ref{m1}) with $t$ replaced by $f(t)$.

\subsection{Remark concerning co-prime solutions}\label{subrem}

Let us also note that each non-trivial solution $f, g$ of (\ref{maincong}) can be extended to an infinite family. Indeed, if the pair $(f, g)$ solves (\ref{maincong}) and we define
$$
f_{1}=f(t), g_{1}=g(t),\quad f_{n}=g_{n-1},\quad g_{n}=\frac{g_{n-1}^3+1}{f_{n-1}},
$$
then the pair of polynomials $(f_{n}, g_{n})$ also solves (\ref{maincong}). Moreover, it is also clear that in some senses the most interesting solutions of the congruence (\ref{maincong}) are those satisfying the condition $(\deg f, \deg g)=1$. Indeed, if additionally $3\nmid \deg f$ then
the corresponding sequence of polynomials generated by the algorithm above produces solutions with co-prime degrees. In this direction one can ask
the following

\begin{ques}
Does there exists infinitely many non-trivial solutions $(F, G)$ of (\ref{maincong}) such that there is no solution $(f, g)$ of (\ref{maincong}) and that for each $n \geq 2$ we have $(F,G)\neq (f_{n}, g_{n})$? In other words, does the congruence (\ref{maincong}) has infinitely many disjoint orbits of non-trivial solutions?
\end{ques}

\section{Quartic case}\label{sec3}
In this section we consider the system of congruences
\begin{equation}\label{Qmod}
f^4+4\equiv 0\pmod{g},\quad 4-4g^2\equiv 0\pmod{f}.
\end{equation}
Let $m=\deg f$ and $n=\deg g.$ We have that $m\leq 2n$ and $n\leq 4m.$
Let
$$
\bar{A}=(a_{0}, a_{1}, \ldots, a_{m-2}, a_{m}), \quad \bar{B}=(b_{0},
b_{1}, \ldots, b_{n-1}, b_{n})
$$
be the vectors of variables. We define the $I_{i}=I_{i}(\bar{A},\bar{B}), i=0,\ldots n-1$ and $J_{i}(\bar{A},\bar{B}), i=0,\ldots, m-1$ as the numerators of the coefficients in the remainders of divisibility of $f^4+4\pmod{g}$ and $4-4g^2\pmod{f}$ respectively. More precisely,
\begin{align*}
	(f(t)^4+4) &\pmod*{g(t)}=\sum_{i=0}^{n-1}F_{i}(\bar{A},\bar{B})t^{i},\\
     (4-4g(t)^2)&\pmod*{f(t)}=\sum_{i=0}^{m-1}G_{i}(\bar{A},\bar{B})t^{i}.
\end{align*}
Let us also note that in the sequel we can assume that $a_{m}>0, b_{n}>0$, indeed, if $(f,g)$ is a solution of (\ref{Qmod}) then $(\pm f, \pm g)$ is also a solution.

To determine the pairs of polynomials satisfying the system \eqref{Qmod} we need to study rational solutions of the system
\begin{equation*}
	R(m,n):\;\begin{cases}\begin{array}{cc}
			F_{i}(\bar{A},\bar{B})=0, & \quad i=0,\ldots, n-1, \\
			F_{j}(\bar{A},\bar{B})=0, & \quad j=0,\ldots, m-1.
		\end{array}
	\end{cases}
\end{equation*}

As in the case of cubic case we perform case by case analysis for $\deg f\leq 2$. Moreover, at the end of the section we characterize solutions of
(\ref{Qmod}) under assumption $\deg f=3, \deg g=2$.

All remarks concerning symbolic computations mentioned in cubic case are in order here too.

All the systems considered in this section are presented in the full form in Appendix B. 

\subsection{The case $\deg f=1$}

Without loss o generality we can assume that $f(t)=t$. As a consequence
we see that $g(t)$ is a divisor of
$$
t^4+4=(t^2-2t+2)(t^2+2t+2).
$$
Thus, the polynomial $g$ takes the form: $g(t)=b_{2}(t^2-2t+2)$ or $g(t)=b_{2}(t^2+2t+2)$ or $g(t)=b_{4}(t^4+4)$. A quick inspection reveals
that $b_{2}=\pm \frac{1}{2}$ and $b_{4}=\pm \frac{1}{4}$. Thus, the only non-trivial solutions are
\begin{equation}\label{m1Q}
\begin{array}{ll}
  f(t)=t, & g(t)=\frac{1}{2}(t^2+2t+2), \\
  f(t)=t, & g(t)=\frac{1}{4}(t^2+2t+2)(t^2-2t+2).
\end{array}
\end{equation}
The corresponding values of $L$ are: $L(t)=-2$ and $L(t)=-t^3/4$. Note that if we replace $t$ by $2t$ then the corresponding polynomials have integer coefficients.

\subsection{The case $\deg f=2, \deg g =1$}

Because $\deg (4-4g(t)^2)=2$ we do not need to consider the whole system. It is enough to note that the part corresponding to the condition $4-4g(t)^2\equiv 0\pmod*{f(t)}$ takes the form
$$
G_{0}=4 \left(a_2 b_0^2-a_0 b_1^2-a_2\right), G_{1}=-8b_{0}b_{1}.
$$
Thus $b_{0}=0$ and the vanishing of $f(t)^4+4$ at $t=0$ implies $a_{0}^4+4=0$ - a contradiction.

\subsection{The case $\deg f=2, \deg g =2$}

We consider the system $R(2,2)$. Let $G(2,2)$ be the Gr\"{o}bner basis of
the ideal generated by the polynomials defining the system $R(2,2)$. The set $G(2,2)$ contains polynomials $H_{i}, i=1, \ldots, 33$. We have
$$
H_{1}=b_1 b_2^7 \left(2 b_1^4+5 b_2^2\right),\quad H_{12}=b_2^4 \left(125 a_2^4+320 b_2^4+512 b_0 b_1^2 b_2^3\right).
$$
Vanishing of $H_{1}$ implies that $b_{1}=0$. Then the equation $H_{12}=0$ is equivalent with $5 b_2^4\left(25 a_2^4+64 b_2^4\right)=0$ and thus $b_{2}=0$ - a contradiction.

Summing up: there is no solutions of the system (\ref{quartic}) with $\deg f=2, \deg g=2$.

\subsection{The case $\deg f=2, \deg g=3$}

We consider the system $R(2,3)$. Let $G(2,3)$ be the Gr\"{o}bner basis of
the ideal generated by the polynomials defining the system $R(2,3)$. The set $G(2,3)$ contains polynomials $H_{i}, i=1, \ldots, 630$. We have
\begin{align*}
H_{1}=&b_2 b_3^6 \left(256 b_2^6-5 b_3^4\right) \left(256 b_2^{18}+15775 b_3^4 b_2^{12}+42400 b_3^8 b_2^6+32000 b_3^{12}\right),\\
H_{23}=&b_3^6 \left(1376 b_2 b_3 b_1^4-1312 b_2^3 b_1^3+9500 b_2 b_3^2 b_1+3820 b_0 b_3^3-10885 b_2^3 b_3\right)
\end{align*}
The unique rational solution of the system $H_{1}=H_{23}=0$ with respect to $b_{2}, b_{0}$ is $b_{0}=b_{2}=0$. Then, with $b_{0}, b_{2}$ chosen in this way we get that the equation $H_{22}=0$ reduces to
$$
1343488 b_3^7 \left(b_1^3+10 b_3\right) \left(2 b_1^6+10 b_3 b_1^3+25 b_3^2\right)=0
$$
and thus $b_{3}=-b_{1}^{3}/10$. Consequently, the equation $H_{201}=0$ reduces to
$$
Cb_1^{21} \left(-10 a_2 b_1^2+25 a_2^2+2 b_1^4\right) \left(10 a_2 b_1^2+25 a_2^2+2 b_1^4\right)=0,
$$
where $C\in\Q\setminus\{0\}$. Thus $b_{1}=0$. However, under all these equalities we get that the equation $H_{628}=0$ reduces to $5a_{2}^3=0$ - a contradiction.

Summing up: there is no solutions of the system (\ref{quartic}) with $\deg f=2, \deg g=3$.

\subsection{The case $\deg f=2, \deg g=4$}

In this section we look for polynomial solutions of the system $R(2,4)$. Let us note that the set of rational solutions of the system $R(2,4)$ is the set theoretic sum of the set of rational solutions of the systems $R_{1}(2,4)$ and $R_{2}(2,4)$, where
\begin{align*}
S_{1}(2,4):\;F_{0}=F_{1}=F_{2}=G_{0}=G_{1,1}=0,\\
S_{2}(2,4):\;F_{0}=F_{1}=F_{2}=G_{0}=G_{1,2}=0,
\end{align*}
where $G_{1,1}=a_2 b_1-a_0 b_3$ and  $G_{1,2}=G_{1}/G_{1,1}$.

Let $G_{i}(2,4)$ be the Gr\"{o}bner basis of the set of polynomials defining the systems $S_{i}(2,4)$ for $i=1, 2$.

We consider the case $i=1$. Then $G_{1}(2,4)$ contains polynomials $H_{i}, i=1, \ldots, 588$. We have $H_{1}=b_3 b_4^6, H_{12}=b_1^7 b_4^5$ and
\begin{align*}
H_{15}=&b_4^6(b_2^4-8 b_0 b_4 b_2^2+16 b_0^2 b_4^2-20 b_4^2)(b_2^4-8 b_0 b_4 b_2^2+16 b_0^2 b_4^2-5 b_4^2),\\
H_{395}=&-b_4^2 \left(25 a_2^8 b_3^2-25 a_2^8 b_2 b_4+256 b_2 b_4^5\right).
\end{align*}
The unique rational solution of the system $H_{1}=H_{12}=H_{395}=0$
is $b_{1}=b_{2}=b_{3}=0$. However, with $b_{1}, b_{2}, b_{3}$ chosen in this way, the equation $H_{15}=0$ reduces to $4(4 b_0^2-5)(16 b_0^2-5)b_4^{10}=0$ and thus $b_{4}=0$ - a contradiction.

Let us consider the case $i=2$. The set $G_{2}(2,4)$ contains polynomials $H_{i}, i=1, \ldots, 1092$. We have
$$
H_{1}=b_4^5 \left(b_3^8+540 b_4^6\right) \left(b_3^{16}+25 b_4^{12}\right)
$$
and hence $b_{4}=0$ - a contradiction.

Hence, we do not get a new non-trivial solution of the system (\ref{quartic}) with $\deg f=2, \deg g=4$.

\subsection{The case $\deg f=2, \deg g=5$}

The condition $G_{1}=0$ implies that
\begin{equation}\label{b0b125}
b_{0}=\frac{a_0 a_2 b_2-a_0^2 b_4}{a_2^2}\quad\mbox{or}\quad b_{1}=\frac{a_0 a_2 b_3-a_0^2 b_5}{a_2^2}.
\end{equation}

After the substitution of the expression for $b_{0}$ and necessary simplifications we get the system $R_{1}(2,5)$. Let $G_{1}(2,5)$ be the Gr\"{o}bner basis of the set of polynomials defining $R_{1}(2,5)$ computed with respect to the standard order $b_{5}<b_{4}<b_{3}<b_{2}<b_{1}<a_{2}<a_{0}$. The set $G_{1}(2,5)$ consists of polynomials $H_{i}, i=1,\ldots, 1011$. In particular, we have
\begin{align*}
H_{1}=&b_4 b_5^4 \left(4096 b_4^{30}+3130855 b_5^8 b_4^{20}-8716800 b_5^{16} b_4^{10}+8192000 b_5^{24}\right),\\
H_{34}=&b_5^4 \left(-22 b_2 b_4^3+33 b_3^2 b_4^2-64 b_2 b_3 b_5 b_4+50 b_2^2 b_5^2\right)
\end{align*}
The equation $H_{1}=0$ implies $b_{4}=0$ and then from $H_{34}=0$ we get that $b_{2}=0$. With $b_{2}, b_{4}$ chosen in this way we get the
equation $H_{23}=0$ reduces to
$$
6656 b_5^5 \left(b_3^5+40 b_5^3\right) \left(b_3^{10}+30 b_5^3 b_3^5+156250 b_5^6\right)=0,
$$
We thus get that $b_5=40^3 u^5$ and $b_3=-40^2 u^3$ for some $u\in\Q$. However, substituting the computed values of $b_2, b_{3}, b_{4}, b_{5}$ into the equation $F_{0}=0$ we obtain that it reduces to the equation $2^{36}\cdot 5^{12}\left(a_0^2-2 a_0+2\right) \left(a_0^2+2 a_0+2\right) u^{20}=0$. Thus $u=0$ and this leads to $b_{5}=0$ - a contradiction.

Next, we turn to the second possibility in (\ref{b0b125}). After the substitution of the expression for $b_{1}$ from (\ref{b0b125} and necessary simplifications we get the system $R_{2}(2,5)$. Let $G_{2}(2,5)$ be the Gr\"{o}bner basis of the set of polynomials defining $R_{1}(2,5)$ computed with respect to the standard order $b_{6}<b_{5}<b_{4}<b_{3}<b_{2}<b_{0}<a_{2}<a_{0}$. The set $G_{2}(2,5)$ consists of polynomials $H_{i}, i=1,\ldots, 705$. In particular, we have $H_{1}=b_5^4(4096b_4^{10}-5b_5^8)$.
Hence the vanishing of $H_{1}$ implies that $b_{5}=0$ - a contradiction.

Summing up: there is no solutions of the system (\ref{quartic}) with $\deg f=2, \deg g=5$.

\subsection{The case $\deg f=2, \deg g=6$}

We consider the system $R(2,6)$ and are looking for its rational solutions. The condition $G_{1}=0$ implies that
\begin{equation}\label{b0b126}
b_{0}=\frac{a_0^3 b_6-a_2 a_0^2 b_4+a_2^2 a_0 b_2}{a_2^3} \quad\mbox{or}\quad b_{1}=\frac{a_0 a_2 b_3-a_0^2 b_5}{a_2^2}.
\end{equation}

After the substitution of the expression for $b_{0}$ and necessary simplifications we get the system $R_{1}(2,5)$. Let $G_{1}(2,6)$ be the Gr\"{o}bner basis of the set of polynomials defining $R_{1}(2,6)$ computed with respect to the standard order $b_{6}<b_{5}<b_{4}<b_{3}<b_{2}<b_{1}<a_{2}<a_{0}$. The set $G_{1}(2,6)$ consists of polynomials $H_{i}, i=1,\ldots, 628$.  In particular, we have $H_{1}=b_6^3(b_5^{12}+10 b_6^{10})$. Thus $H_{1}=0$ implies that $b_{6}=0$ - a contradiction.

Next we turn to the second possibility in (\ref{b0b126}). After the substitution of the expression for $b_{1}$ from (\ref{b0b126} and necessary simplifications we get the system $R_{2}(2,6)$. Let $G_{2}(2,6)$ be the Gr\"{o}bner basis of the set of polynomials defining $R_{2}(2,6)$ computed with respect to the standard order $b_{6}<b_{5}<b_{4}<b_{3}<b_{2}<b_{0}<a_{2}<a_{0}$. The set $G_{2}(2,6)$ consists of polynomials $H_{i}, i=1,\ldots, 337$. In particular, we have
\begin{align*}
H_{1}=&b_5 b_6^4,\\
H_{4}=&-b_6^3(b_4 b_5-b_3 b_6),\\
H_{9}=&-b_6^3 \left(b_4^6-9 b_2 b_6 b_4^4+27 b_2^2 b_6^2 b_4^2+10 b_6^4-27 b_2^3 b_6^3\right),\\
H_{96}=&b_6^4 \left(5 a_2^4-8 b_4^2+24 b_2 b_6\right)
\end{align*}
The system $H_{1}=H_{4}=H_{96}=0$ implies that
Solving the system $H_{1}=H_{4}=H_{96}=0$ we get that
$$
b_{3}=b_{5}=0,\quad b_{2}=\frac{8 b_4^2-5 a_2^4}{24 b_6}.
$$
However, with $b_{3}, b_{5}, b_{2}$ chosen in this way, the equation $H_{96}=0$ reduces to the form
$$
-\frac{5}{512}b_6^3 \left(25 a_2^{12}+1024 b_6^4\right)=0.
$$
Hence $b_{6}=0$ - a contradiction.

Summing up: there is no solutions of the system (\ref{quartic}) with $\deg f=2, \deg g=6$.

\subsection{The case $\deg f=2, \deg g=7$}

This case is easy. Indeed, if $g\in\Q[t]$ is of degree 7, then the condition $f(t)^4+4\equiv 0\pmod*{g(t)}$ implies the equality $f(t)^4+4=g(t)L(t)$, where $\deg L=1$. Thus, if $r\in\Q$ is a root of $L(t)=0$, then
$f(r)$ is a rational root of the polynomial $t^4+4$ - a contradiction.

\subsection{The case $\deg f=2, \deg g=8$}

This case is equally easy. Indeed, if $g\in\Q[t]$ is of degree 8, then the condition $f(t)^4+4\equiv 0\pmod*{g(t)}$ implies the equality $f(t)^4+4=cg(t)$ for some $c\in\Q\setminus\{0\}$. Because $f(t)=a_{2}^2+a_{0}$, thus $b_{1}=b_{3}=b_{5}=b_{7}=0$ and comparing coefficients on both sides of the equality $f(t)^4+4=cg(t)$ we get that
$$
b_{0}=\frac{a_0^4+4}{c},\quad b_{2}=\frac{4 a_0^3 a_2}{c},\quad b_{4}=\frac{6 a_0^2 a_2^2}{c},\quad b_{6}=\frac{4 a_0a_2^3}{c},\quad b_{8}=\frac{a_2^4}{c}.
$$
The condition $4-4g(t)^2\equiv 0\pmod*{f(t)}$ implies that $c=\pm 4$. Thus, we get that $g(t)=\pm \frac{1}{2}(f(t)^2+f(t)+2)$. However, this is trivial solution coming form the first solution in (\ref{m1Q}) with $t$
replaced by $f(t)$.

\subsection{Some remarks concerning the case of $\deg f=3$}

From our computations in previous subsections we find that there are no
non-trivial solutions of (\ref{quartic}) in the case of $\deg f=2$. A question arises whether there are any nontrivial solution with $\deg f>2$.
Although the resulting systems of polynomials are more and more complicated, this natural question motivated us to look for new solutions. It is clear that the most natural candidate is the next simplest one, i.e., $\deg f=3$. Unfortunately we were unable to find all solutions of the congruence (\ref{quartic}) in this case but we were able to cover the case $\deg f=3, \deg g=2$. More precisely, we prove the following

\begin{thm}\label{f3g2}
The only solution $f, g\in\Q[t]$ of the congruence (\ref{quartic}) satisfying  $\deg f=3, \deg g=2$ is
$$
f(t)=4t(t^2-t+1),\quad g(t)=2t^2-2t+1.
$$
The corresponding value of $L$ is $L(t)=8 t^7-16 t^6+28 t^5-20 t^4+14 t^3+t+1$.
\end{thm}
\begin{proof}
The subsystem of the system $R(3,2)$ corresponding to the condition $4-4g(t)^2\equiv 0\pmod*{f(t)}$ has the form
\begin{align*}
G_{0}=&(a_3 b_0^2-2 a_0 b_1 b_2-a_3=0,\\
G_{1}=&-a_0 b_2^2-2 a_1 b_1 b_2+2 a_3 b_0 b_1=0,\\
G_{2}=&a_3 b_1^2-a_1 b_2^2+2 a_3 b_0 b_2=0.
\end{align*}
It can be easily solved with respect to $b_{0}, a_{0}, a_{1}$. The solutions are
\begin{equation}\label{sol132}
b_{0}=\frac{-2 b_1^2-b_2}{b_2},\quad a_{0}=\frac{2 a_3 b_1 \left(b_1^2+b_2\right)}{b_2^3},\quad a_{1}=\frac{-3 a_3 b_1^2-2a_3 b_2}{b_2^2}
\end{equation}
or
\begin{equation}\label{sol232}
b_{0}=\frac{b_2-2 b_1^2}{b_2},\quad a_{0}=-\frac{2 a_{3}b_{1}\left(b_2-b_1^2\right)}{b_2^3},\quad a_{1}=\frac{2a_3b_2-3 a_3 b_1^2}{b_2^2}.
\end{equation}

Let us consider first (\ref{sol132}). Substituting computed values of $b_{0}, a_{0}, a_{1}$ into the equations coming from the condition $f(t)^4+4\equiv 0\pmod*{g(t)}$ we get the system
$$
F'_{0}=a_3^4 b_2^2+15 a_3^4 b_1^2 b_2+27 a_3^4 b_1^4+4 b_2^8=0,\quad F'_{1}=-3 a_3^4 b_1 b_2^4(9 b_1^2+2 b_2)=0.
$$
The only solutions of interest is
$$
b_{2}=-\frac{9 b_1^2}{2},\quad a_{3}=\pm \frac{27 b_1^3}{2}.
$$
Up to the sign and replacement of $t$ by $(6t-2)/9b_{1}$ we get exactly the solution presented in the statement of our theorem.

Exactly the same solution is obtained in the case (\ref{sol232}) and thus
we omit the simple details.
\end{proof}

We finish with the following

\begin{ques}
Does the congruence (\ref{quartic}) has infinitely many non-trivial solutions in polynomials $f, g\in\Q[t]$?
\end{ques}

We expect that the answer is yes.

\section{Diophantine problems related to cyclic cubic and quartic fields}\label{sec4}

\subsection{A different look on the congruence $f^3+g^3+1\equiv 0\pmod*{fg}$}

In case of cubic case we deal with the congruence $f^3+g^3+1\equiv 0\pmod*{fg}$. As we mentioned in the introduction, there are infinitely many solutions of this congruence (both in integers and in polynomials with integer coefficients). However, our approach was to find solutions without worrying about the corresponding value of $(f^3+g^3+1)/fg$. This motivated us to concentrate on the value of $d$ and asks about integer points on the curve
$$
H_d:\quad x^3+y^3+dxy+1=0.
$$
Let us note that for $d\neq -3$, the curve $H_{d}$ is a genus one curve and thus has only finitely many integer solutions. Thus, for any given $d$, the problem concerning characterization of the set $H_{d}(\Z)$ is interesting and non-trivial. Moreover, the curve $H_{d}$ is the so-called Hessian form of an elliptic curve. One can prove that over some extension of $\Q$, an elliptic curve with torsion point of order 3, is birationally equivalent with $H_{d}$ for some (not necessarily rational) value of $d$.

\begin{thm}
If $(x,y)\in\mathbb{Z}^2$ is a solution of equation $H_d$ for some $d\neq
-3$ and $d_1$ is an integer dividing $d^3+27,$ then we have
\begin{align*}
x&=\frac{3d  d_{1} + 3 \, d_{1}^{2} \pm \sqrt{-27d^{2} d_{1}^{2} - 18 \, d d_{1}^{3} - 3 \, d_{1}^{4} - 12\left(d^{3} + 27\right)d_{1}}}{18d_{1}},\\
y&=\frac{d+d_{1}-3x}{3}.
\end{align*}
\end{thm}
\begin{proof}
We have that
$$
27( x^3+y^3+dxy+1)-(d^{2} + 3 \, d x + 9 \, x^{2} + 3 \, d y - 9 \, x y +
9 \, y^{2})(3x+3y-d)=d^3+27.
$$
Let $d_1$ be a divisor of $d^3+27.$ We obtain that
\begin{eqnarray*}
3x+3y-d&=&d_1,\\
d^{2} + 3 \, d x + 9 \, x^{2} + 3 \, d y - 9 \, x y + 9 \, y^{2}&=&\frac{-d^3-27}{d_1}.
\end{eqnarray*}
The first equation gives the formula for $y,$ that is $y=\frac{1}{3} \,
d + \frac{1}{3} \, d_{1} - x.$ Using this expression we eliminate $y$ from the second equation and the solutions for $x$ are coming from a simple quadratic equation.
\end{proof}

The proof of the theorem provides an algorithm to determine all the integral solutions, we implemented the method in SageMath \cite{sage} and applied it for all values in the range $[-2\cdot 10^6, 2\cdot 10^6]$ except $d=-3.$ The complete list of integral solutions can be downloaded from \url{https://tengely.github.io/Hd.pdf}. We note that for any $d$ we have the trivial solutions $(-1,0),(0,-1).$ Also, if $d$ is a square, let say $d=s^2$ for some integer $s,$ then $(\pm s, -1),(-1,\pm s)$ are solutions. Hence there are many equations having at least 6 solutions (in the given range 333).
There are 2755 cases in the range $[-2\cdot 10^6, 2\cdot 10^6]$ having non-trivial solutions (solutions other than the above mentioned $(-1,0),(0,-1)$). There are 1317 equations with 4 solutions, 1 equation with 5 solutions, 1413 equations with 6 solutions, 13 equations with 8 solutions, 7
equations with 10 solutions, 3 equations with 12 solutions and 1 equation
with 14 solutions. The table containing those $d$ in the considered range such that $X_{d}$ has at least 10 solutions is presented below. The parallel SageMath code (using 3 cores of a Core I5 desktop machine) was running for 3 days and 8 hours.

\begin{equation*}
\begin{array}{|l|l|}
\hline
\hline
d & solutions (x,y)\\
\hline
\hline
-1056 & \left(1063, -1406\right), \left(-1406, 1063\right), \left(38, -201\right), \left(-201, 38\right), \left(7, -86\right), \\
        & \left(-86, 7\right), \left(4, -65\right), \left(-65, 4\right), \left(0, -1\right), \left(-1, 0\right) \\
\hline
25 & \left(27, -19\right), \left(-19, 27\right), \left(5, -1\right), \left(-1, 5\right), \left(0, -1\right), \left(-1, 0\right), \left(-1, -5\right),\\
     &\left(-5, -1\right), \left(-2, -7\right), \left(-7, -2\right), \left(-9, -13\right), \left(-13, -9\right) \\
\hline
49 & \left(19, -7\right), \left(-7, 19\right), \left(7, -1\right), \left(-1, 7\right), \left(0, -1\right), \left(-1, 0\right), \left(-1, -7\right), \\
     &\left(-7, -1\right), \left(-7, -18\right), \left(-18, -7\right) \\
\hline
225 & \left(161, -93\right), \left(-93, 161\right), \left(26, -3\right), \left(-3, 26\right), \left(15, -1\right), \left(-1, 15\right), \\
      &\left(0, -1\right), \left(-1, 0\right), \left(-1, -15\right), \left(-15, -1\right), \left(-63, -109\right), \left(-109, -63\right) \\
\hline
1369 & \left(1159, -733\right), \left(-733, 1159\right), \left(37, -1\right), \left(-1, 37\right), \left(0, -1\right), \left(-1, 0\right), \\
       &\left(-1, -37\right), \left(-37, -1\right), \left(-373, -657\right), \left(-657, -373\right), \left(-437, -691\right),  \\
       &\left(-691, -437\right)\\
\hline
1609 & \left(1267, -772\right), \left(-772, 1267\right), \left(468, -133\right), \left(-133, 468\right), \left(0, -1\right), \left(-1, 0\right), \\
       &\left(-11, -133\right), \left(-133, -11\right), \left(-133, -457\right), \left(-457, -133\right), \left(-331, -693\right),  \\
       &\left(-693, -331\right), \left(-365, -721\right), \left(-721, -365\right) \\
\hline
9801 & \left(7791, -4771\right), \left(-4771, 7791\right), \left(99, -1\right), \left(-1, 99\right), \left(0, -1\right), \left(-1, 0\right), \\
       &\left(-1, -99\right), \left(-99, -1\right), \left(-2989, -4881\right), \left(-4881, -2989\right) \\
\hline
43089 & \left(172333, -158006\right), \left(-158006, 172333\right), \left(6697, -1037\right), \left(-1037, 6697\right),  \\
        &\left(6135, -871\right), \left(-871, 6135\right), \left(0, -1\right), \left(-1, 0\right), \left(-16226, -22583\right),  \\
        &\left(-22583, -16226\right)\\
\hline
66049 & \left(53793, -33361\right), \left(-33361, 53793\right), \left(257, -1\right), \left(-1, 257\right), \left(0, -1\right),  \\
        &\left(-1, 0\right), \left(-1, -257\right), \left(-257, -1\right), \left(-20511, -33073\right), \left(-33073, -20511\right) \\
\hline
212521 & \left(27729, -3610\right), \left(-3610, 27729\right), \left(461, -1\right), \left(-1, 461\right), \left(0, -1\right),  \\
         &\left(-1, 0\right), \left(-1, -461\right), \left(-461, -1\right), \left(-59, -3541\right), \left(-3541, -59\right) \\
\hline
455625 & \left(367667, -226929\right), \left(-226929, 367667\right), \left(675, -1\right), \left(-1, 675\right), \left(0, -1\right), \\
         & \left(-1, 0\right), \left(-1, -675\right), \left(-675, -1\right), \left(-140529, -227683\right),  \\
         &\left(-227683, -140529\right)\\
\hline
\end{array}
\end{equation*}
\begin{center}
Table 1. Values of $d\in[-2\cdot 10^6,2\cdot 10^6]$ such that $H_{d}$ has at least 10 solutions together with the complete set of integer solutions.
\end{center}

\subsection{A genus three curve of Balady and Washington}

Let $T\in\Z$ and consider the curve 
$$
X_T:\quad (x+y)^4-4x^2y^2+Txy(x+y)+4=0.
$$
Balady and Washington showed a clear connection between integer points on $X_{4T}$ with $2T\in\Z$ and cyclic quartic fields generated by the polynomial
\begin{equation}\label{quarpol}
X^4+(2s^3+Ls^2-4ps+2Lp)X^3-(3s^2+3Ls-6p)X^2+2LX+1.
\end{equation}
In fact, each integer point on $X_{4T}$ leads to suitable polynomial of the above form with square discriminant. For most values of $T$, the curve $X_{T}$ represents genus 3 curve. Thus, famous Faltings theorem implies that the set of rational (and thus integer) points on $X_{T}$ is finite. Balady and Washington observed that $X_{-8}$ is a singular curve and were able to prove the existence of infinitely many quartic number fields with cyclic Galois group of order 4 and such that the roots of the defining polynomial (\ref{quarpol}) generate the group of units or a subgroup of index 5. They also noted some sporadic values of $T$ such that $X_{T}$ has non-trivial integer points, i.e., different then $(\pm 1, \mp 1)$. It is an interesting question whether there are infinitely many values of $T\in\Z$ such that on $X_{T}$ we have non-trivial integer points. Similarly, as in the case of the family $H_{d}$ we investigated this question from a computational point of view. More precisely, we have the following result.

\begin{thm}
Let $T$ be an even integer such that $T\neq \pm 8.$ If $(x,y)\in\mathbb{Z}^2$ is a solution of $X_T$ and $d_1$ is a divisor of
$$
\frac{(T^2+64)(T-8)(T+8)}{16},
$$
then we have that
$$
(4x+4y)^2=d_1+\frac{(T^2+64)(T-8)(T+8)}{16d_1}-\frac{T^2}{2}.
$$
\end{thm}
\begin{proof}
We have that
\begin{eqnarray*}
&& \left(8x^2+32xy+8y^2-2Tx-2Ty+\frac{T^2}{4}\right)\left(8x^2+8y^2+2Tx+2Ty+\frac{T^2}{4}\right)=\\
&& =\frac{(T^2+64)(T-8)(T+8)}{16}.
\end{eqnarray*}
Since $T$ is even we obtain that
$$
\left(8x^2+32xy+8y^2-2Tx-2Ty+\frac{T^2}{4}\right)=d_1
$$
and
$$
\left(8x^2+8y^2+2Tx+2Ty+\frac{T^2}{4}\right)=\frac{(T^2+64)(T-8)(T+8)}{16d_1},
$$
where $d_1$ divides $\frac{(T^2+64)(T-8)(T+8)}{16}.$ The statement follows by taking the sum of these equations.
\end{proof}

Based on the above theorem we have two strategies to determine the complete set of integral solutions of $X_T.$ If we can factorize the integer
$\frac{(T^2+64)(T-8)(T+8)}{16},$ then we collect together the divisors $d_1$ for which the number
$$d_1+\frac{(T^2+64)(T-8)(T+8)}{16d_1}-\frac{T^2}{2}$$
is a square divisible by 16. After that the problem can be reduced to determine integral roots of univariate polynomials. If the factorization of the integer $\frac{(T^2+64)(T-8)(T+8)}{16}$ is too expensive, then we may
compute the resultant of the polynomials
$$
(x+y)^4-4x^2y^2+Txy(x+y)+4
$$
and
$$
(4x+4y)^2-d_1-\frac{(T^2+64)(T-8)(T+8)}{16d_1}+\frac{T^2}{2}
$$
with respect to $y.$ Let us denote it by $R_{d_1,T}(x).$ This is a degree
8 reducible polynomial that is a product of two degree 4 polynomials. Therefore we can provide bound for $|x|$ as a function of $d_1$ and $T$ and finally we note that we can bound $d_1$ as well.
\begin{cor}\label{C4}
Let $T\neq 8$ be an even integer for which $2\leq T\leq 15\cdot 10^6.$  The Diophantine equation $X_T$ has at least 3 integral solutions in the following cases
\begin{center}
\begin{tabular}{|c|c|}
\hline
$T$ & solutions $(x,y)$ \\
\hline
40 & $(\pm 1, \mp 1), (-1,-5), (-5,-1)$ \\
\hline
154 & $(\pm 1, \mp 1), (-5,37), (37,-5)$ \\
\hline
77236 & $(\pm 1, \mp 1), (-5,629), (629,-5)$ \\
\hline
\end{tabular}.
\end{center}
\end{cor}
In the study of cyclic quartic fields one needs to consider the Diophantine equation $X_{4L},$ where $L$ or $2L$ is an integer. Balady and Washington \cite{BalWas4} considered the special cases $L=-2$ and $L=2$ and they also made a computer search and provided a table containing 5 examples (with $L=\pm 2$ and $L=-77/2$). The solutions in Corollary \ref{C4} yield the following polynomials
\begin{center}
\begin{tabular}{|c|c|c|c|}
\hline
$x$ & $y$ & $L$ & Polynomial \\
\hline
$-1$ & $-5$ & 10 & $t^4 + 148t^3 + 102t^2 + 20t + 1$ \\
\hline
$-5$ & 37 & $77/2$ & $t^4 + 114395t^3 - 7878t^2 + 77t + 1$ \\
\hline
$-5$ & 629 & 19309 & $t^4 + 7890798742t^3 - 37333446t^2 + 38618t + 1$ \\
\hline
\end{tabular}
\end{center}
We note that the example with $L=77/2$ corresponds to the one with $L=-77/2$ in the table of Balady and Washington by $t\mapsto -t.$

Based on the polynomial solutions of the congruences, see the formulas at \eqref{m1Q} and in Theorem 3.1, we determined integral points on $X_{2t^3}$ such that $1\leq t\leq 80000$ and on $X_{8t^7 - 16t^6 + 28t^5 - 20t^4 + 14t^3 + t + 1}$ with  $1\leq t\leq 1001.$ There are only trivial solutions. 

Our computations clearly show that the problem of finding values of $T\neq 8$ such that $X_{T}$ contains non-trivial integer points, is difficult. We thus formulate the following explicit problem.

\begin{prob}
Find more integers $T$ such that $X_{T}$ contains non-trivial integer points. 
\end{prob}

\begin{rem}
{\rm Let us note that each integer point $(x,y)$ on $X_{T}$ for some $T$, leads to the simultaneous congruences
\begin{equation}\label{systemcong}
x^4+4\equiv 0\pmod*{y},\quad y^4+4\equiv 0\pmod*{x},\quad 4-4x^2y^2\equiv 0\pmod*{x+y}.
\end{equation}
However, there are integer solutions of the system (\ref{systemcong}) which do not lead to the solutions of the equation defining $X_{T}$ for any integer $T$, i.e., the value
$$
T=T(x,y)=\frac{(x+y)^4-4x^2y^2+4}{xy(x+y)}
$$
is not an integer. We performed small numerical search for solutions of (\ref{systemcong}) satisfying $x\leq y$ with $|x|\leq 10^6$ and which do not leading to integer value of $T(x,y)$. Let us note that if $(x, y)$ is a solution and $x, y$ are both positive/negative then $(y, x)$ is also a solution. Moreover, if $x<0, y>0$ and $(x, y)$ is a solution of (\ref{systemcong}) with corresponding value of $T$, then $(-y, -x)$ is also a solution with corresponding value $-T$. Thus in our computations we deal only with $y>|x|$. The results of our computations are presented in the table below.
\begin{equation*}
\begin{array}{|l|l|l|l|}
\hline
 x & y & x+y & T(x,y) \\
\hline
 -20 & 26 & 6 & \frac{1385}{4} \\
 -20 & 34 & 14 & \frac{761}{4} \\
 -4 & 10 & 6 & \frac{85}{4} \\
 -2 & 4 & 2 & \frac{59}{4} \\
 2 & 2 & 4 & \frac{49}{4} \\
 2 & 4 & 6 & \frac{87}{4} \\
 2 & 10 & 12 & \frac{319}{4} \\
 4 & 26 & 30 & \frac{983}{4} \\
 10 & 122 & 132 & \frac{7393}{4} \\
 148 & 4322 & 4470 & \frac{556227}{4} \\
 \hline
\end{array}
\end{equation*}
\begin{center}
Table 2. Solutions $(x,y)$ of (\ref{systemcong}) satisfying the conditions: $|x|<10^7, |x|<y$ and $T(x,y)\not\in\Z$.
\end{center}

In the light of our computations we state the following 
\begin{ques}
Does the system (\ref{systemcong}) has infinitely many solutions in integers $x, y$?   
\end{ques}
}
\end{rem}

\bigskip

Let us observe that there is natural map (of degree 2)
$$
\phi:\; X_{T}\ni (x,y)\mapsto (x+y,xy)=(X,Y)\in Q_{T},
$$
where $Q_{T}:\;X^4-4Y^2+TXY+4=0$. Thus, it is also natural to ask about values of $T$ which lead to curves $Q_{T}$ with non-trivial integer points. Similarly, as in the case of the curve $X_{T}$ one can propose an efficient method which for a given value of $T$ allows to find all integer points on $Q_{T}$. In fact, we can work with slightly more general form
$$
Q_{a,b}:\quad x^4-4y^2+axy+b=0.
$$
Note that $Q_{T,4}=Q_{T}$ and recall that $Q_{4a,4}$ were studied in \cite{BalWas4}.

In case of the equation $Q_{a,b}$ we have the following result.
\begin{thm}
Let $a,b$ be integers such that $a^4\neq 1024b.$
Let $g=\gcd(32,8a,a^2).$
If $(x,y)\in\mathbb{Z}^2$ is a solution of $Q_{a,b}$ and $d_1$ is a divisor of
$$
\frac{a^4-1024b}{g^2},
$$
then we have that
$$
64x^2=d_1g+\frac{a^4-1024b}{d_1g}-2a^2.
$$
\end{thm}
\begin{proof}
We have that
$$
\frac{1}{g^2}\left(32x^2-8ax+64y+a^2\right)\left(32x^2+8ax-64y+a^2\right)=\frac{a^4-1024b}{g^2}.
$$
We obtain that
$$
\frac{1}{g}\left(32x^2-8ax+64y+a^2\right)=d_1
$$
and
$$
\frac{1}{g}\left(32x^2+8ax-64y+a^2\right)=\frac{a^4-1024b}{d_1g^2},
$$
where $d_1$ divides $\frac{a^4-1024b}{g^2}.$ The statement follows by taking the sum of these equations.

%
\end{proof}

We determined all solutions of $Q_{4a,4}$ for $3\leq a\leq 5\cdot 10^6$ by using SageMath, the complete list of integral solutions can be downloaded from \url{https://tengely.github.io/Qab.pdf}.
It is interesting to note that among the 552 cases that yield non-trivial
(solutions other than $(0,\pm 1)$) there are 544 with 6 solutions and 8 with 10 solutions. In what follows we list the latter 8 cases.
We also considered the equation $Q_{2a,4}$ for odd values of $a$ such that $3\leq a \leq 5\cdot 10^6,$
since in case of the equation $Q_{4a,4}$ those values are also interesting for which $2a$ is an integer. In the given range there are 111 equations having non-trivial solutions. In all these cases there are 6 integral solutions. The computation was performed on a Core I5 desktop computer and
we used 3 cores parallel in SageMath, it took 23.3 hours to complete. We provide the list of equations that have 10 integral solutions in the table below.

\begin{equation*}
\begin{array}{|l|l|}
\hline
U & \mbox{integral solutions of}\;x^4-4y^2+4axy+4=0 \\
\hline
250 & (0, 1), \left(156, 42485\right), \left(-10, 1\right), \left(-156, 3485\right), \left(10, -1\right), \left(-10, -2501\right), \\
    &\left(0, -1\right),\left(10, 2501\right), \left(156, -3485\right), \left(-156, -42485\right) \\
\hline
691 & \left(0, 1\right), \left(112, 77897\right), \left(-112, 505\right), \left(-24, -16589\right), \left(24, 16589\right),\\
    & \left(112, -505\right), \left(-112, -77897\right), \left(-24, 5\right), \left(0, -1\right), \left(24, -5\right) \\
\hline
6750 & \left(-30, -202501\right), \left(0, 1\right), \left(-2142, 355265\right), \left(30, -1\right), \left(2142, 14813765\right), \\
     & \left(-30, 1\right), \left(-2142, -14813765\right), \left(30, 202501\right), \left(2142, -355265\right), \left(0, -1\right)] \\
\hline
13718 & \left(0, 1\right), \left(-38, 1\right), \left(38, -1\right), \left(7964, -8538133\right), \left(38, 521285\right),\\
      &\left(7964, 117788285\right), \left(-38, -521285\right), \left(0, -1\right), \left(-7964, -117788285\right), \\
      &\left(-7964, 8538133\right) \\
\hline
19924 & \left(0, 1\right), \left(-4216, -84929585\right), \left(-904, 9265\right), \\
      &\left(-4216, 930001\right), \left(-904, -18020561\right), \left(904, 18020561\right), \left(4216, -930001\right), \\
      &\left(904, -9265\right), \left(4216, 84929585\right), \left(0, -1\right) \\
\hline
50510 & \left(0, 1\right), \left(-4556, -230590685\right), \left(4556, -467125\right), \left(-258, 85\right), \\
      &\left(-4556, 467125\right), \left(258, 13031665\right), \left(4556, 230590685\right),  \\
      &\left(-258, -13031665\right), \left(258, -85\right), \left(0, -1\right) \\
\hline
236788 & \left(0, 1\right), \left(23212, 5509495865\right), \left(-23212, -5509495865\right), \\
       &\left(-23212, 13172809\right), \left(659512, -152988152785\right),   \\
       &\left(-659512, 152988152785\right), \left(23212, -13172809\right),  \\
       &  \left(659512, 309152680241\right), \left(0, -1\right), \left(-659512, -309152680241\right)\\
\hline
715822 & (-413444, -318861253093), (413444, -22908942125), (142, -1), \\
       &(-142, -101646725), (-413444, 22908942125), (0, -1), (0, 1), \\
       &(413444, 318861253093), (-142, 1), (142, 101646725)\\
\hline
\end{array}
\end{equation*}
\begin{center} Table 3. Values of $a\in [0, 5\cdot 10^6]$ such that $Q_{4a,a}$ has 10 integral solutions  \end{center}

\begin{acknowledgement}
The authors are very grateful to Attila Peth{\H{o}} for bringing the problem to their attention.
\end{acknowledgement}

\noindent {\bf Appendix A}

The polynomials defining the system $S(m,n)$ for $m=2, n\in\{1,\ldots, 6\}$.

The case $(m, n)=(2, 2)$.
\begin{align*}
F_{0}=&a_0^3 b_2^5-3 a_0^2 a_2 b_0 b_2^4+3 a_0 a_2^2b_0^2b_2^3-a_2^3b_0^3b_2^2-3a_0a_2^2b_0b_1^2b_2^2+3a_2^3b_0^2b_1^2b_2+\\
      &-a_2^3b_0 b_1^4+b_2^5, \\
F_{1}=&a_2 b_1(a_2^2 b_1^4+3 a_0 a_2 b_2^2 b_1^2-4 a_2^2b_0 b_2 b_1^2+3
a_0^2 b_2^4-6 a_0 a_2 b_0 b_2^3+3 a_2^2 b_0^2 b_2^2),\\
G_{0}=&a_2^3 b_0^3-3 a_0 a_2^2 b_0 b_1^2-3 a_0 a_2^2 b_0^2 b_2+3 a_0^2 a_2 b_0 b_2^2+3 a_0^2 a_2 b_1^2b_2-a_0^3 b_2^3+a_2^3, \\
G_{1}=&b_1(3 a_2^2 b_0^2-6 a_0 a_2 b_2 b_0-a_0 a_2 b_1^2+3 a_0^2 b_2^2).
\end{align*}

\bigskip

The case $(m, n)=(2, 3)$.
\begin{align*}
F_{0}=&a_0^3b_3^4+a_2^3 b_0^2 b_3^2+3 a_0 a_2^2 b_0 b_2 b_3^2-2 a_2^3b_0b_1b_2b_3+a_2^3b_0b_2^3+b_3^4,\\
F_{1}=&-a_2^2(b_1 b_2-b_0 b_3)(-a_2 b_2^2-3 a_0 b_3^2+2 a_2 b_1 b_3),\\
F_{2}=&a_2^2 b_2^4+3 a_0 a_2 b_3^2 b_2^2-3 a_2^2 b_1 b_3 b_2^2+2 a_2^2 b_0 b_3^2 b_2+3 a_0^2 b_3^4-3 a_0a_2b_1b_3^3+a_2^2b_1^2b_3^2,\\
G_{0}=&3a_0^4b_2b_3^2-a_2a_0^3b_2^3-3a_2a_0^3b_0b_3^2-6a_2a_0^3b_1b_2b_3+3a_2^2a_0^2b_0b_2^2+3a_2^2a_0^2b_1^2b_2+\\
      &6a_2^2a_0^2b_0b_1b_3-3a_2^3a_0b_0b_1^2-3a_2^3a_0b_0^2b_2+a_2^4b_0^3+a_2^4,\\
G_{1}=&(a_2b_1-a_0b_3)\times \\
      &(-a_0^3b_3^2+3a_2a_0^2 b_2^2+2 a_2 a_0^2 b_1 b_3-a_2^2 a_0 b_1^2-6
a_2^2 a_0 b_0 b_2+3 a_2^3 b_0^2).
\end{align*}

\bigskip
The case $(m, n)=(2, 4)$.
\begin{align*}
F_{0}=&-a_2^3b_0b_3^2+a_2^3 b_0 b_2 b_4-3 a_0 a_2^2 b_0 b_4^2+a_0^3 b_4^3+b_4^3,\\
F_{1}=&a_2 b_1 b_3^2-a_2 b_0 b_4b_3+3 a_0 b_1 b_4^2-a_2 b_1 b_2 b_4,\\
F_{2}=&3 a_0^2 b_4^3-a_2^2 b_0 b_4^2-3 a_0 a_2 b_2 b_4^2+a_2^2 b_2^2 b_4+a_2^2 b_1 b_3b_4-a_2^2 b_2 b_3^2,\\
F_{3}=&a_2 b_3^3+3 a_0 b_4^2 b_3-2 a_2 b_2 b_4 b_3+a_2 b_1 b_4^2,\\
G_{0}=&a_0^6 b_4^3-3 a_2 a_0^5 b_2 b_4^2-3 a_2 a_0^5 b_3^2 b_4+3 a_2^2 a_0^4 b_2 b_3^2+3 a_2^2 a_0^4 b_0 b_4^2\\
      &+3 a_2^2 a_0^4 b_2^2 b_4+6a_2^2a_0^4b_1b_3b_4-a_2^3 a_0^3 b_2^3-3 a_2^3 a_0^3 b_0 b_3^2-6 a_2^3 a_0^3 b_1 b_2 b_3+\\
      &-3 a_2^3 a_0^3 b_1^2b_4-6a_2^3a_0^3b_0b_2 b_4+3 a_2^4 a_0^2 b_0 b_2^2+3 a_2^4 a_0^2 b_1^2 b_2+6 a_2^4 a_0^2 b_0 b_1 b_3\\
      &+3 a_2^4 a_0^2 b_0^2 b_4-3 a_2^5 a_0 b_0 b_1^2-3 a_2^5a_0 b_0^2 b_2+a_2^6 b_0^3+a_2^6,\\
G_{1}=&(a_2b_1-a_0b_3)(3a_0^4b_4^2-a_2a_0^3b_3^2-6a_2a_0^3b_2b_4+3a_2^2a_0^2b_2^2+2a_2^2a_0^2b_1b_3+\\
      &6a_2^2a_0^2b_0 b_4-a_2^3a_0b_1^2-6a_2^3a_0b_0b_2+3a_2^4b_0^2).
\end{align*}

\bigskip

The case $(m, n)=(2, 5)$.
\begin{align*}
F_{0}=&a_0^3 b_5^2+a_2^3 b_0 b_4+b_5^2,\quad F_{1}=b_1 b_4-b_0 b_5,\quad F_{2}=a_2^2 b_2 b_4-a_2^2 b_1 b_5+3 a_0^2 b_5^2,\\
F_{3}=&b_3 b_4-b_2 b_5,\quad F_{4}=a_2 b_4^2+3 a_0 b_5^2-a_2 b_3 b_5,\\
G_{0}=&-3 a_0^7 b_4 b_5^2+a_2 a_0^6 b_4^3+3 a_2 a_0^6 b_2 b_5^2+6 a_2 a_0^6 b_3 b_4 b_5-3 a_2^2 a_0^5 b_2 b_4^2+\\
      &-3 a_2^2 a_0^5 b_0 b_5^2-3 a_2^2 a_0^5 b_3^2 b_4-6 a_2^2 a_0^5 b_2
b_3 b_5-6 a_2^2 a_0^5 b_1 b_4 b_5+3 a_2^3 a_0^4 b_2 b_3^2\\
      &+3 a_2^3 a_0^4 b_0 b_4^2+3 a_2^3 a_0^4 b_2^2 b_4+6 a_2^3 a_0^4 b_1
b_3 b_4+6 a_2^3 a_0^4 b_1 b_2 b_5+6 a_2^3 a_0^4 b_0 b_3 b_5+\\
      &-a_2^4 a_0^3 b_2^3-3 a_2^4 a_0^3 b_0 b_3^2-6 a_2^4 a_0^3 b_1 b_2 b_3-3 a_2^4 a_0^3 b_1^2 b_4-6 a_2^4 a_0^3 b_0 b_2 b_4+\\
      &-6 a_2^4 a_0^3 b_0 b_1 b_5+3 a_2^5 a_0^2 b_0 b_2^2+3 a_2^5 a_0^2 b_1^2 b_2+6 a_2^5 a_0^2 b_0 b_1 b_3+3 a_2^5 a_0^2 b_0^2 b_4+\\
      &-3 a_2^6 a_0 b_0 b_1^2-3 a_2^6 a_0 b_0^2 b_2+a_2^7 b_0^3+a_2^7,\\
G_{1}=&(a_0^2 b_5-a_2 a_0 b_3+a_2^2 b_1)(-a_0^5b_5^2+3 a_2 a_0^4 b_4^2+2 a_2 a_0^4 b_3 b_5-a_2^2 a_0^3 b_3^2+\\
      &\quad -6 a_2^2 a_0^3 b_2 b_4-2 a_2^2 a_0^3 b_1 b_5+3 a_2^3 a_0^2 b_2^2+2 a_2^3 a_0^2 b_1 b_3+6 a_2^3 a_0^2 b_0 b_4+\\
      &\quad -a_2^4 a_0 b_1^2-6 a_2^4 a_0 b_0 b_2+3 a_2^5 b_0^2).
\end{align*}

\noindent {\bf Appendix B}

The polynomials defining the system $R(m,n)$ for $m=2, n\in\{2,\ldots, 9\}$ and $(m, n)=(3,2)$.

The case $(m,n)=(2,2)$.
\begin{align*}
F_{0}=&a_0^4 b_2^7-4 a_0^3 a_2 b_0 b_2^6+6 a_0^2 a_2^2 b_0^2 b_2^5-4 a_0 a_2^3 b_0^3 b_2^4-6 a_0^2a_2^2 b_0 b_1^2 b_2^4+a_2^4 b_0^4 b_2^3\\
      &+12 a_0 a_2^3 b_0^2 b_1^2 b_2^3-4 a_0 a_2^3 b_0 b_1^4b_2^2-6 a_2^4
b_0^3 b_1^2 b_2^2+5 a_2^4 b_0^2 b_1^4 b_2-a_2^4 b_0 b_1^6+4 b_2^7,\\
F_{1}=&b_1(a_2b_1^2+2 a_0 b_2^2-2 a_2 b_0 b_2) \times\\
      &(a_2^2 b_1^4+2 a_0 a_2 b_2^2 b_1^2-4 a_2^2 b_0 b_2b_1^2+2 a_0^2 b_2^4-4 a_0 a_2 b_0 b_2^3+2 a_2^2 b_0^2 b_2^2),\\
G_{0}=&-4 a_2^2 b_0^2+4 a_0 a_2b_1^2+8 a_0 a_2 b_0 b_2-4 a_0^2 b_2^2+5 a_2^2,\\
G_{1}=&b_1 \left(a_2 b_0-a_0 b_2\right)
\end{align*}

\bigskip
The case $(m, n)=(2, 3)$.
\begin{align*}
F_{0}=&a_0^4 b_3^6+4 a_0 a_2^3 b_0^2 b_3^4+6 a_0^2 a_2^2 b_0 b_2 b_3^4-2 a_2^4 b_0^2 b_1 b_3^3-8 a_0 a_2^3 b_0 b_1 b_2 b_3^3\\
      &+4 a_0 a_2^3 b_0 b_2^3 b_3^2+3 a_2^4 b_0^2 b_2^2 b_3^2+3 a_2^4 b_0
b_1^2 b_2 b_3^2-4 a_2^4 b_0 b_1 b_2^3 b_3+a_2^4 b_0 b_2^5+4 b_3^6,\\
F_{1}=&(b_1 b_2-b_0 b_3)\times \\
      &\; (a_2^2 b_2^4+4 a_0 a_2 b_3^2 b_2^2-4 a_2^2 b_1 b_3 b_2^2+2 a_2^2 b_0 b_3^2 b_2+6 a_0^2 b_3^4-8 a_0 a_2 b_1 b_3^3+3 a_2^2 b_1^2 b_3^2),\\
F_{2}=&a_2^3 b_2^6+4 a_0 a_2^2 b_3^2 b_2^4-5 a_2^3 b_1 b_3 b_2^4+4 a_2^3 b_0 b_3^2 b_2^3+6 a_0^2 a_2 b_3^4 b_2^2-12 a_0 a_2^2 b_1 b_3^3 b_2^2\\
      &+6 a_2^3 b_1^2 b_3^2 b_2^2+8 a_0 a_2^2 b_0 b_3^4 b_2-6 a_2^3 b_0 b_1 b_3^3 b_2+4 a_0^3 b_3^6-6 a_0^2 a_2 b_1 b_3^5+a_2^3 b_0^2 b_3^4\\
      &+4 a_0 a_2^2 b_1^2 b_3^4-a_2^3 b_1^3 b_3^3,\\
G_{0}=&4 a_0^3 b_3^2-4 a_2 a_0^2 b_2^2-8 a_2 a_0^2 b_1 b_3+4 a_2^2 a_0 b_1^2+8 a_2^2 a_0 b_0 b_2-4 a_2^3 b_0^2+5 a_2^3,\\
G_{1}=&(a_2 b_0-a_0 b_2)(a_2 b_1-a_0 b_3).
\end{align*}
\bigskip

The case $(m, n)=(2, 4)$.
\begin{align*}
F_{0}=&a_0^4 b_4^5-6 a_0^2 a_2^2 b_0 b_4^4+a_2^4 b_0^2 b_4^3+4 a_0 a_2^3 b_0 b_2 b_4^3-a_2^4 b_0 b_2^2 b_4^2-4 a_0 a_2^3 b_0 b_3^2 b_4^2+\\
      &-2 a_2^4 b_0 b_1 b_3 b_4^2+3 a_2^4 b_0 b_2 b_3^2 b_4-a_2^4 b_0 b_3^4+4 b_4^5,\\
F_{1}=&a_2^2 b_1 b_3^4-a_2^2 b_0 b_4 b_3^3+4 a_0 a_2 b_1 b_4^2 b_3^2-3 a_2^2 b_1 b_2 b_4 b_3^2-4 a_0 a_2 b_0 b_4^3 b_3\\
      &+2 a_2^2 b_1^2 b_4^2 b_3+2 a_2^2 b_0 b_2 b_4^2 b_3+6 a_0^2 b_1 b_4^4-2 a_2^2 b_0 b_1 b_4^3-4 a_0 a_2 b_1 b_2 b_4^3+a_2^2 b_1 b_2^2 b_4^2,\\
F_{2}=&4 a_0^3 b_4^5-4 a_0 a_2^2 b_0 b_4^4-6 a_0^2 a_2 b_2 b_4^4+a_2^3 b_1^2 b_4^3+4 a_0 a_2^2 b_2^2 b_4^3+2 a_2^3 b_0 b_2 b_4^3\\
      &+4 a_0 a_2^2 b_1 b_3 b_4^3-a_2^3 b_2^3 b_4^2-a_2^3 b_0 b_3^2 b_4^2-4 a_0 a_2^2 b_2 b_3^2 b_4^2-4 a_2^3 b_1 b_2 b_3 b_4^2+a_2^3 b_1 b_3^3 b_4\\
      &+3 a_2^3 b_2^2 b_3^2 b_4-a_2^3 b_2 b_3^4,\\
F_{3}=&a_2^2 b_3^5+4 a_0 a_2 b_4^2 b_3^3-4 a_2^2 b_2 b_4 b_3^3+3 a_2^2 b_1 b_4^2 b_3^2+6 a_0^2 b_4^4 b_3-2 a_2^2 b_0 b_4^3 b_3+\\
      &-8 a_0 a_2 b_2 b_4^3 b_3+3 a_2^2 b_2^2 b_4^2 b_3+4 a_0 a_2 b_1 b_4^4-2 a_2^2 b_1 b_2 b_4^3,\\
G_{0}=&a_0^4 b_4^2-a_2 a_0^3 b_3^2-2 a_2 a_0^3 b_2 b_4+a_2^2 a_0^2 b_2^2+2 a_2^2 a_0^2 b_1 b_3+2 a_2^2 a_0^2 b_0 b_4-a_2^3 a_0 b_1^2+\\
      &-2 a_2^3 a_0 b_0 b_2+a_2^4 b_0^2-a_2^4,\\
G_{1}=&(a_2 b_1-a_0 b_3)(a_0^2 b_4-a_2 a_0 b_2+a_2^2 b_0).
\end{align*}
\bigskip

The case $(m,n)=(2,5)$.
\begin{align*}
F_{0}=&a_2^4 b_0 b_4^3+a_2^4 b_0 b_2 b_5^2-2 a_2^4 b_0 b_3 b_4 b_5+4 a_0 a_2^3 b_0 b_4 b_5^2+a_0^4 b_5^4+4 b_5^4,\\
F_{1}=&a_2 b_1 b_4^3-a_2 b_0 b_5 b_4^2+4 a_0 b_1 b_5^2 b_4-2 a_2 b_1 b_3 b_5 b_4-4 a_0 b_0 b_5^3\\
      &+a_2 b_1 b_2 b_5^2+a_2 b_0 b_3 b_5^2,\\
F_{2}=&4 a_0^3 b_5^4-4 a_0 a_2^2 b_1 b_5^3+a_2^3 b_2^2 b_5^2+a_2^3 b_1 b_3 b_5^2+a_2^3 b_0 b_4 b_5^2+4 a_0 a_2^2 b_2 b_4 b_5^2+\\
      &-a_2^3 b_1 b_4^2 b_5-2 a_2^3 b_2 b_3 b_4 b_5+a_2^3 b_2 b_4^3,\\
F_{3}=&-a_2 b_3 b_4^3+a_2 b_2 b_5 b_4^2-a_2 b_1 b_5^2 b_4-4 a_0 b_3 b_5^2 b_4+2 a_2 b_3^2 b_5 b_4+a_2 b_0 b_5^3\\
      &+4 a_0 b_2 b_5^3-2 a_2 b_2 b_3 b_5^2,\\
F_{4}=&a_2^2 b_4^4+4 a_0 a_2 b_5^2 b_4^2-3 a_2^2 b_3 b_5 b_4^2+2 a_2^2 b_2 b_5^2 b_4+6 a_0^2 b_5^4-a_2^2 b_1 b_5^3+\\
      &-4 a_0 a_2 b_3 b_5^3+a_2^2 b_3^2 b_5^2,\\
G_{0}=&4 a_0^5 b_5^2-4 a_2 a_0^4 b_4^2-8 a_2 a_0^4 b_3 b_5+4 a_2^2 a_0^3 b_3^2+8 a_2^2 a_0^3 b_2 b_4+8 a_2^2 a_0^3 b_1 b_5+\\
      &-4 a_2^3 a_0^2 b_2^2-8 a_2^3 a_0^2 b_1 b_3-8 a_2^3 a_0^2 b_0 b_4+4
a_2^4 a_0 b_1^2+8 a_2^4 a_0 b_0 b_2-4 a_2^5 b_0^2+5 a_2^5,\\
G_{1}=&(a_0^2 b_4-a_2 a_0 b_2+a_2^2 b_0)(a_0^2 b_5-a_2 a_0 b_3+a_2^2 b_1).
\end{align*}
\bigskip

The case $(m, n)=(2,6)$.
\begin{align*}
F_{0}=&a_0^4 b_6^3-4 a_2^3 a_0 b_0 b_6^2-a_2^4 b_0 b_5^2+a_2^4 b_0 b_4 b_6+4 b_6^3,\\
F_{1}=&a_2b_1 b_5^2-a_2 b_0 b_6 b_5+4 a_0 b_1 b_6^2-a_2 b_1 b_4 b_6,\\
F_{2}=&-a_2^3b_2b_5^2-a_2^3 b_0 b_6^2+a_2^3 b_2 b_4 b_6+a_2^3 b_1 b_5 b_6-4 a_0 a_2^2 b_2b_6^2+4 a_0^3 b_6^3,\\
F_{3}=&a_2 b_3 b_5^2-a_2 b_2 b_6 b_5+a_2 b_1 b_6^2+4 a_0 b_3b_6^2-a_2 b_3 b_4 b_6,\\
F_{4}=&6 a_0^2 b_6^3-a_2^2 b_2 b_6^2-4 a_0 a_2 b_4 b_6^2+a_2^2b_4^2 b_6+a_2^2 b_3 b_5 b_6-a_2^2 b_4 b_5^2,\\
F_{5}=&a_2 b_5^3+4 a_0 b_6^2 b_5-2 a_2b_4 b_6 b_5+a_2 b_3 b_6^2,\\
G_{0}=&a_0^6 b_6^2-a_2 a_0^5 b_5^2-2 a_2 a_0^5 b_4 b_6+a_2^2 a_0^4 b_4^2+2 a_2^2 a_0^4 b_3 b_5+2 a_2^2 a_0^4 b_2 b_6-a_2^3 a_0^3 b_3^2+\\
      &-2 a_2^3 a_0^3 b_2 b_4-2 a_2^3 a_0^3 b_1 b_5-2 a_2^3 a_0^3 b_0 b_6+a_2^4 a_0^2 b_2^2+2 a_2^4 a_0^2 b_1 b_3\\
      &+2 a_2^4 a_0^2 b_0 b_4-a_2^5 a_0 b_1^2-2 a_2^5 a_0 b_0 b_2+a_2^6 b_0^2-a_2^6,\\
G_{1}=&(a_0^2 b_5-a_2 a_0 b_3+a_2^2 b_1)(-a_0^3b_6+a_2 a_0^2 b_4-a_2^2 a_0 b_2+a_2^3 b_0).
\end{align*}
\bibliography{all}

\vskip 1cm

\noindent Szabolcs Tengely, Institute of Mathematics, University of Debrecen, P.O.Box 400, 4010 Debrecen, Hungary. email:\;{\tt tengely@science.unideb.hu}\\
\bigskip

\noindent Maciej Ulas, Jagiellonian University, Faculty of Mathematics and Computer Science, Institute of
Mathematics, {\L}ojasiewicza 6, 30-348 Krak\'ow, Poland; email: {\tt maciej.ulas@uj.edu.pl}

\end{document}